\documentclass[11pt,a4paper]{article}
\usepackage[utf8]{inputenc}
\usepackage{amsmath,amssymb,amsthm}
\usepackage{geometry}
\usepackage{enumitem}
\usepackage{mathrsfs}
\usepackage{booktabs}  
\usepackage{bm}
\usepackage{tikz-cd}
\usepackage[mathscr]{euscript}
\usepackage{authblk,orcidlink}   
\usepackage{romannum}     
\usepackage{comment}
\usepackage{float}
\usepackage{hyperref}

\usepackage[backend=biber,        
            style=numeric,        
            sorting=nty,         
            doi=true,             
            isbn=false,           
            url=false,             
            giveninits=true      
            ]{biblatex}
\DeclareFieldFormat[article,incollection,inproceedings,inbook]{title}{#1\isdot}   
\DeclareFieldFormat{booktitle}{\mkbibemph{#1}\isdot}

\newtheorem{theorem}{Theorem}[section]
\newtheorem{lemma}[theorem]{Lemma}
\newtheorem{proposition}[theorem]{Proposition}
\newtheorem{corollary}[theorem]{Corollary}
\newtheorem{definition}[theorem]{Definition}
\newtheorem{remark}[theorem]{Remark}
\newtheorem{example}[theorem]{Example}

\title{Counting Invariants for Torus Quotients under FPF Group Actions}
\author{Xue Zhang}
\date{}

\begin{document}
\pagenumbering{arabic} 

\maketitle

\footnotetext{{\itshape E-mail address:} \texttt{zhangxue@cau.edu.cn}}
\footnotetext{Department of Applied Mathematics, China Agricultural University, Beijing 100083, China}

\begin{abstract}
Let $G$ be a finite group admitting a rational fixed-point-free (FPF) representation $V$, with a $G$-invariant full-rank lattice $L\subset V$. We study the induced $G$-action on the complex torus $V_{\mathbb{C}}/L$ and the quotient orbifold $V_{\mathbb{C}}/(G\ltimes L)$. We prove that every FPF group admits a unique irreducible rational FPF representation. We derive closed-form formulas for three core geometric invariants: the number of non-free points of the torus action, the number of singular points of the quotient orbifold, and its orbifold Euler characteristic. 
\end{abstract}

\textbf{2020 Mathematics Subject Classification:} 20C10 (primary), 20H15, 57R18 (secondary)


\section*{Introduction}
Finite fixed-point-free (FPF) groups are well known from the classification of spherical space-form groups. These are precisely the finite groups admitting an FPF complex representation $\rho$, i.e., $\det(\rho(g)-\mathrm{id})\neq 0$ for every non-identity $g\in G$. Their structure is completely understood: every FPF group satisfies the so-called $pq$-condition, and falls into one of six families. The classification of FPF groups was completed by Vincent, Zassenhaus, and Wolf, and is recorded in Wolf's monograph \cite{wolf2011spaces}. For a modern survey on the spherical space-form problem, see Hambleton \cite{hambleton2015s}.

The action of the Galois group $\operatorname{Gal}(\overline{\mathbb{Q}}/\mathbb{Q})$ on the irreducible complex characters of $G$ guarantees the existence of rational FPF representation of an FPF group \cite[Thm. 70.15]{curtis1966rep}. As our first main result, we establish the following property of rational FPF representations:
\begin{theorem}\label{thm:int01u}
Let $G$ be an FPF group. Then $G$ admits a unique irreducible rational FPF representation up to isomorphism.
\end{theorem}

Classical work focuses on FPF actions on spheres, whereas comparatively little attention has been paid to FPF group actions on complex tori. In this paper we study the analogous problem on complex tori constructed from rational representations. Let $L\subset V$ be a $G$-invariant full-rank lattice in a rational FPF representation $V,$ and let $V_{\mathbb{C}}/L$ be the associated complex torus. Since $\det(g-\mathrm{id})\neq 0$ for $g\neq e,$ each non-identity element has finitely many fixed points, and the quotient orbifold $V_{\mathbb{C}}/(G\ltimes L)$ has only finitely many singular points. Such quotients are flat orbifolds, and their structure is intimately connected with flat manifolds and Bieberbach groups \cite{charlap1986b}. Quotients of complex tori by finite groups acting freely in codimension two have recently been characterized by a Chern-class vanishing condition \cite{claudon2022n,graf2020}, building on the projective case treated in \cite{greb2016}. For cyclic $G$, Fujiki \cite{Fujiki1974onr} studied resolutions of cyclic quotient singularities, and \cite{fulton1993int,pham2012} treat toric resolutions of such quotients. The orbifold Euler characteristic is a basic numerical invariant of the quotient.

The main goal of this paper is to count the geometric invariants of an FPF action on the torus $V_{\mathbb{C}}/L.$ Using M\"{o}bius inversion on the subgroup lattice of $G$ (Rota \cite{rota1964}), the number of non-free points of $G$-action is
\begin{equation*}
|F| = -\sum_{\{e\}<H\leq G} \mu_G\bigl(\{e\}, H\bigr)\cdot  |\operatorname{Fix}(H)|.
\end{equation*}
On the one hand, we note an important property of FPF groups: all non-cyclic subgroups $H$ of $G$ satisfy $\mu_G(\{e\},H)=0$ (Lemma~\ref{lemmuvsh34}). This reduces the above formula to counting only the cyclic subgroups of $G.$ On the other hand, 
$|\operatorname{Fix}(H)|$ depends only on the order of $H$ and equals $\Phi_{d}(1)^{\dim_{\mathbb{Q}}V/\varphi(d)},$ where $d=|H|$ and $\Phi_d$ denotes the $d$-th cyclotomic polynomial (Lemma~\ref{lem33detai}). Combining these two facts, we obtain closed formulas for three quantities that depend only on $\dim_{\mathbb{Q}}V$ and $G,$ and not on the choice of the lattice $L$:
\begin{itemize}
  \item the number of non-free points of the $G$-action on the torus (Theorem~\ref{thmnff35});
  \item the number of singular points of the quotient (Theorem~\ref{thmpif36});
  \item the orbifold Euler characteristic of the quotient (Corollary~\ref{cor:eulerg37}).
\end{itemize}

As examples, we carry out explicit computations for cyclic groups and generalized quaternion groups, since  the counts of their cyclic subgroups are easily obtained. We then give closed-form expressions depending only on the prime divisors of $m.$ Finally, we determine the cyclic subgroup counts for all FPF groups of Types \Romannum{1}--\Romannum{6}, this makes the formulas of Section~\ref{sec:acts} fully explicit for every FPF group.

The paper is organized as follows. Section~\ref{sec:1classfpfg} reviews the classification of FPF groups from Wolf \cite{wolf2011spaces} and recalls some crucial properties. Section~\ref{sec:2intfpfg} introduces integral FPF representations and establishes the structural theorems mentioned above. Section~\ref{sec:acts} proves the
counting formulas via M\"{o}bius-function techniques. Section~\ref{sec4:exams} works out the cyclic and generalized quaternion examples. Section~\ref{sec5:cycfpf} computes the cyclic subgroup counts for all six types of groups.

\textit{Acknowledgments.} The author acknowledges the use of DeepSeek V4 in the preparation of this manuscript. Specifically, DeepSeek V4 was used for language polishing and for generating illustrative examples to verify Theorem \ref{thm:fpfweiyi}. All mathematical arguments and conclusions presented in this work were independently completed by the author. The author take full intellectual responsibility for the content of this paper.

\section{Classification of FPF groups}\label{sec:1classfpfg}

In this section, we review the classification of finite groups admitting FPF representations, following the treatment in Wolf \cite{wolf2011spaces}. This classification is fundamental to the solution of the spherical space form problem, as established by Vincent, Zassenhaus, and Wolf.

\begin{definition}
Let $G$ be a group and let $V$ be a finite-dimensional complex vector space. A complex representation $\rho\colon G \to\operatorname{GL}(V)$ is called FPF if for every non-identity element $g \in G,$ $\det(\rho(g) - \mathrm{id})\ne 0,$ equivalently, the origin of $V/G$ is an isolated singularity. A finite group $G$ is called an FPF group if it admits an FPF complex representation.

A complex character $\chi$ of $G$ is said to be \textit{FPF} if $\chi$ is the character of some FPF complex representation of $G$.
\end{definition}

For example, every finite subgroup of $\operatorname{SU}(2)$ is FPF via its natural $\mathbb{C}^2$ representation. Below we give a useful lemma for the FPF condition.

\begin{lemma}\label{lem:21fpfpbf}
Let $\chi$ be a complex character of $G$. Then $\chi$ is FPF if and only if for every prime $p$ and every element $g\in G$ of order $p$,
\begin{equation}\label{lem21fpfs}
\sum_{j=0}^{p-1}\chi(g^j)=0.   
\end{equation}
\end{lemma}
\begin{proof}
Let $g$ be an element of order $p$. The eigenvalues $\lambda_1,\dots,\lambda_d$ of $\rho(g)$ are all $p$-th roots of unity, where $d=\chi(1)$. Then
\[\sum_{j=0}^{p-1}\chi(g^j)=\sum_{i=1}^d\sum_{j=0}^{p-1}\lambda_i^j=\#\{i\colon \lambda_i=1\}\cdot p.\]
Thus \eqref{lem21fpfs} holds if and only if $g$ has no eigenvalue equal to $1$. Now, a non-identity element has eigenvalue $1$ if and only if some element of prime order has eigenvalue $1$. Hence the assertion follows.
\end{proof}

Before presenting the full classification of FPF groups, we state two key properties of FPF groups, proved by character theory and Sylow subgroup analysis in Wolf’s monograph \cite{wolf2011spaces}.

\begin{theorem}\label{thm22pqsylow}
Let $G$ be an FPF group. The following conditions, combining Theorem~5.3.1 and Theorem~5.3.2 of \cite{wolf2011spaces}, hold for $G$:
\begin{enumerate}[label=(\arabic*)]
\item ($pq$-condition) Every subgroup of $G$ with order $pq$ for arbitrary primes $p,q$ is cyclic. 
\item (Sylow subgroup restriction) For any odd prime $p,$ all Sylow $p$-subgroups of $G$ are cyclic; the Sylow $2$-subgroup of $G$ is either cyclic or a generalized quaternion group $Q_{2^a}$ with integer $a\geq 3.$
\end{enumerate}
\end{theorem}

\begin{remark}
The generalized quaternion group $Q_{4m}$ of order $4m$ ($m\ge 2$) is defined by generators and relations
\[Q_{4m}=\langle a,b\mid a^{2m}=1, b^2=a^m,bab^{-1}=a^{-1}\rangle.\]
In Theorem~\ref{thm22pqsylow}(2), $Q_{2^a}$ denotes the special case $m=2^{a-2}.$
\end{remark}

FPF groups fall into six types, classified according to their Sylow-subgroup structure: Types \Romannum{1}-\Romannum{4} are solvable, and Types \Romannum{5}-\Romannum{6} are non-solvable, built upon the binary icosahedral group $\operatorname{SL}_2(\mathbb{F}_5).$ The complete list is as follows.

\begin{theorem}[Vincent, Zassenhaus, Wolf, cf. {\cite[Chapter~6]{wolf2011spaces}}]\label{thm:FPF_classifi}
A finite group $G$  is an FPF group if and only if $G$ is isomorphic to one of the following six types.

\medskip 

\noindent\textbf{Type \Romannum{1}.} $G$ has generators $A$ and $B$ with relations
\[A^m = B^n = 1, \quad BAB^{-1} = A^r,\]
where $m,n \geq 1,$ $((r-1)n, m) = 1,$ and $r^n \equiv 1 \pmod{m}.$ Let $d$ be the order of $r$ in $(\mathbb{Z}/m\mathbb{Z})^\times.$ Then $d$ divides $n,$ and $n/d$ is divisible by every prime divisor of $d.$ In this case $G$ has order $mn$ and every Sylow subgroup is cyclic.

\medskip 

\noindent\textbf{Type \Romannum{2}.} $G$ has generators $A,$ $B,$ and $R$ with relations as in Type \Romannum{1} and additionally
\[R^2 = B^{n/2}, \quad RAR^{-1} = A^l, \quad RBR^{-1} = B^k,\]
where $l^2 \equiv r^{k-1} \equiv 1 \pmod{m},$ $n=2^uv,u\geq 2,$ $k\equiv -1 \pmod{2^u},$ and $k^2 \equiv 1 \pmod{n}.$ 
In this case $G$ has order $2mn$ and the Sylow $2$-subgroups are generalized quaternion.

\medskip

\noindent\textbf{Type \Romannum{3}.} $G$ has generators $A,$ $B,$ $P,$ and $Q$ with relations as in Type \Romannum{1} and additionally
\[P^4 = 1, \hspace{0.5em} P^2 = Q^2=(PQ)^2, \hspace{0.5em} AP = PA, \hspace{0.5em} AQ = QA, \hspace{0.5em} 
BPB^{-1} = Q, \hspace{0.5em} BQB^{-1} = PQ,\]
where $n \equiv 3 \pmod{6}.$ In this case $G$ has order $8mn$ and the Sylow $2$-subgroup is $Q_8.$

\medskip

\noindent\textbf{Type \Romannum{4}.} $G$ has generators $A,$ $B,$ $P,$ $Q,$ and $R$ with relations as in Type \Romannum{3} and additionally
\[R^2 = P^2, \quad RPR^{-1} = QP, \quad RQR^{-1} = Q^{-1}, \quad RAR^{-1} = A^l, \quad RBR^{-1} = B^k,\]
where $k^2 \equiv 1 \pmod{n},$ $k \equiv -1 \pmod{3},$ and $r^{k-1}\equiv l^2 \equiv 1 \pmod{m}.$ In this case $G$ has order $16mn$ and the Sylow $2$-subgroup is $Q_{16}.$

\medskip

\noindent\textbf{Type \Romannum{5}.} $G = G_1 \times \operatorname{SL}_2(\mathbb{F}_5),$ where $G_1$ is a group of Type \Romannum{1} with $(|G_1|,30)=1.$

\medskip

\noindent\textbf{Type \Romannum{6}.} $G = \langle G_2, S \rangle,$ where $G_2=G_1 \times \operatorname{SL}_2(\mathbb{F}_5)$ is a normal subgroup of index $2$ and Type \Romannum{5}, $S^2 = -I\in\operatorname{SL}_2(\mathbb{F}_5),$ and $SAS^{-1} = \theta(A)$ for all $A\in \operatorname{SL}_2(\mathbb{F}_5),$ where $\theta$ is the automorphism of $\operatorname{SL}_2(\mathbb{F}_5)$ given by conjugation by $\begin{pmatrix} 0 & -1 \\ 2 & 0 \end{pmatrix}.$ Moreover, $S$ normalizes $G_1.$
\end{theorem}

\begin{remark}
In the above classification, $m$ is necessarily odd in all six types.
\end{remark}

\begin{example}[cf. {\cite[Theorem 5.5.6]{wolf2011spaces}}]\label{exam25type1}
Let $G$ be an FPF group of Type \Romannum{1} with parameters $m, n, r, d, n'$ satisfying $n = n'd.$ For integers $k,l$ with $(k,m) = 1 = (l,n),$ let $\alpha = e^{2\pi i k/m},\beta = e^{2\pi i l/n'}.$ Define 
\[\pi_{k,l}(A) = \begin{pmatrix}
\alpha & & & \\
& \alpha^r & & \\
& & \ddots & \\
& & & \alpha^{r^{d-1}}
\end{pmatrix}, \qquad
\pi_{k,l}(B) = \begin{pmatrix}
0 & 1 & & \\
& 0 & \ddots & \\
& & \ddots & 1 \\
\beta & & & 0
\end{pmatrix}.\]
Then each $\pi_{k,l}$ is an irreducible FPF representation of $G$ of degree $d.$
\end{example}

\section{Rational and integral FPF representations}\label{sec:2intfpfg}
The FPF notion extends naturally from complex to rational and integral representations. As integral representations are special among rational ones, we focus on the integral case.

\begin{definition}
Let $L$ be a lattice. A representation $\rho\colon G\to\operatorname{Aut}(L)$ is called integral FPF if for every non-identity $g\in G,$ $\det(\rho(g)-\mathrm{id})\ne 0.$
\end{definition}

\begin{lemma}\label{lem33detai}
Let $A \in \operatorname{GL}_n(\mathbb{Z})$ have order $m,$ and suppose that for every $1\leq i \leq m-1,$ $A^i$ has no eigenvalue equal to $1.$ Then the characteristic polynomial of $A$ is a power of the $m$-th cyclotomic polynomial:
\[\det(\lambda I_n-A) = \Phi_m(\lambda)^{\frac{n}{\varphi(m)}}.\]
\end{lemma}
\begin{proof}
Since $A^m=I_n,$ all eigenvalues of $A$ are $m$-th roots of unity. If $\lambda$ is an eigenvalue of $A,$ then $\lambda^{i}$ is an eigenvalue of $A^i,$ so $\lambda^i\neq 1$ for all $1\leq i\leq m-1.$ This forces every eigenvalue of $A$ to be a primitive $m$-th root of unity, hence a root of $\Phi_m(x).$ Since $\det(\lambda I_n-A)$ is monic with integer coefficients and $\Phi_m(\lambda)$ is irreducible over $\mathbb{Q},$ every primitive $m$-th root of unity has the same multiplicity as a root of $\det(\lambda I_n-A).$ Comparing degrees gives the exponent $n/\varphi(m)$, completing the proof.
\end{proof}

\begin{lemma}[{Galois Correspondence for Rational Representations \cite[Thm. 70.15]{curtis1966rep}}] \label{lemgalcorr}
Let $K$ be a field of characteristic zero, and let $K[G]$ denote the group algebra of a finite group $G$ over $K.$ Write $\overline{K}$ for an algebraic closure of $K.$ 
\begin{enumerate}
\item Let $M$ be a simple left $K[G]$-module. Choose any irreducible $\overline{K}[G]$ constituent $\rho$ of $M\otimes_K\overline{K}.$ There exists an isomorphism of $\overline{K}[G]$-modules
\begin{equation*}
M\otimes_K\overline{K}\cong \bigoplus_{\sigma\in\operatorname{Gal}(K(\chi_\rho)/K)} s(\chi_\rho)\cdot \rho^\sigma,
\end{equation*}
where $K(\chi_\rho)$ is the character field of $\chi_\rho,$ $s(\chi_\rho)$ is the Schur index of $\rho$ over $K,$ and $\rho^\sigma$ stands for the Galois conjugate representation. Hence
\[\dim_{K}M=s(\rho)\cdot [K(\chi_\rho):K]\cdot \dim_{\overline{K}}\rho.\]

\item Given an irreducible $\overline{K}[G]$-module with character $\chi,$ its Galois orbit 
\[\{\chi^\sigma\mid \sigma\in\operatorname{Gal}(K(\chi)/K)\}\] 
uniquely determines a simple $K[G]$-module $M,$ whose base extension to $\overline{K}$ recovers the direct sum decomposition above. Distinct Galois orbits yield non-isomorphic simple $K[G]$-modules.
\end{enumerate}
\end{lemma}

\begin{proposition}
Every FPF group admits an integral FPF representation.
\end{proposition}
\begin{proof}
Let $(\rho,V)$ be an irreducible FPF complex representation of $G$ of dimension $n$, with character $\chi$, and let $\mathbb{Q}(\chi)$ be the character field of $\rho.$ First, by Lemma \ref{lemgalcorr}, there exists an irreducible $\mathbb{Q}$-representation $\tau$ of $G$ such that
\[\tau \otimes_{\mathbb{Q}} \mathbb{C}
\cong \bigoplus_{\sigma \in \operatorname{Gal}(\mathbb{Q}(\chi)/\mathbb{Q})} s(\chi)\cdot\rho^{\sigma}\]
with dimension $m:=\dim_{\mathbb{Q}} \tau=s(\chi)\cdot [\mathbb{Q}(\chi):\mathbb{Q}] \cdot n,$ where $s(\chi)$ is the Schur index of $\chi$ over $\mathbb{Q}.$ For each $g \in G,$ the eigenvalues of $\tau(g)$ are exactly the Galois conjugates of the eigenvalues of $\rho(g),$ each repeated $s(\chi)$ times. Since the Galois automorphisms fix $1,$ none of these eigenvalues equal $1$ for $g\neq e,$ so $\tau$ is a rational FPF representation of $G.$

Next, Theorem 73.5 of Curtis-Reiner \cite{curtis1966rep} shows that every rational representation of a finite group is $\mathbb{Q}$-equivalent to an integral representation, a fact guaranteed by the existence of $G$-invariant full-rank lattices. Therefore $\tau$ is $\mathbb{Q}$-equivalent to an integral FPF representation of dimension $m.$
\end{proof}

To study the action of $\operatorname{Gal}(\overline{\mathbb{Q}}/\mathbb{Q})$ on the characters of $G$, we decompose $G$ as a semidirect product. Let $C_m$ denote the cyclic group of order $m.$ Every FPF group of Types~\Romannum{1}--\Romannum{5} given in Theorem~\ref{thm:FPF_classifi} admits the following structure: there exists a normal cyclic Hall subgroup $C_m=\langle A\rangle$ such that
\[G=C_m\rtimes_\alpha H,\qquad |G|=m|H|,\qquad (m,2|H|)=1,\]
where $H=G/C_m$ and $\alpha\colon H\to \operatorname{Aut}(C_m)=(\mathbb{Z}/m\mathbb{Z})^\times$ is the action. Concretely,
\[\begin{array}{ll}
\text{Type \Romannum{1}:} & H=\langle B\rangle\cong C_n,\quad \alpha(B)=r;\\[2mm]
\text{Type \Romannum{2}:} & H=\langle B,R\rangle,\quad \alpha(B)=r,\quad \alpha(R)=l;\\[2mm]
\text{Type \Romannum{3}:} & H=\langle B,P,Q\rangle,\quad \alpha(B)=r,\quad \alpha(P)=\alpha(Q)=1;\\[2mm]
\text{Type \Romannum{4}:} & H=\langle B,P,Q,R\rangle,\quad \alpha(B)=r,\quad \alpha(P)=\alpha(Q)=1,\quad \alpha(R)=l. \\[2mm]
\text{Type \Romannum{5}:} & H=\langle B\rangle\times \operatorname{SL}_2(\mathbb{F}_5), \quad \alpha(B)=r, \quad \alpha(\operatorname{SL}_2(\mathbb{F}_5))=1.
\end{array}\]

Recall that a finite group satisfies the $pq$-condition if every subgroup of order $pq$ (primes $p,q$, possibly equal) is cyclic. By Theorem \ref{thm22pqsylow}, every FPF group satisfies the $pq$-condition. 

\begin{lemma}\label{lem:3.1primact}
Let $K=\ker\alpha\leq H.$ Then every element of prime order in $H$ lies in $K$, i.e., it acts trivially on $C_m.$
\end{lemma}
\begin{proof}
Let $x\in H$ be an element of prime order $p$ and suppose $x\notin K$. Then $x$ acts nontrivially on $C_m.$ Hence there exists a prime $q\mid m$ such that $x$ acts nontrivially on the unique subgroup $C_q$ of $C_m,$ where necessarily $p\mid q-1$ and thus $p\neq q$. Then $\langle x, C_q\rangle\cong C_q\rtimes C_p$ is a non-abelian group of order $pq$, which violates the $pq$-condition. Therefore $x\in K$.
\end{proof}

By Mackey’s little-group theory \cite{Mackey1958} for the semidirect product $G$, every irreducible character of $G$ is of the form
\[\chi_{\lambda,\tau}=\operatorname{Ind}_{C_m\rtimes H_\lambda}^{G}\bigl(\lambda\boxtimes\tau\bigr),\]
where $\lambda\in\operatorname{Irr}(C_m)$, $H_\lambda$ is the stabilizer of $\lambda$ under the $H$-action on $\operatorname{Irr}(C_m)$,
 \[(h\cdot\lambda)(x)=\lambda(h^{-1}\cdot x),\]
  and $\tau\in\operatorname{Irr}(H_\lambda)$.
Two such induced characters coincide if and only if the corresponding pairs are $H$-conjugate.

\begin{lemma}\label{lem:26fpfsemp}
Let notation be as above. The irreducible FPF characters of $G$ are precisely
\[\chi_{\lambda,\tau}=\operatorname{Ind}_{C_m\rtimes K}^{G}\bigl(\lambda\boxtimes\tau\bigr),\]
where $\lambda$ runs over the faithful one-dimensional characters of $C_m$, and $\tau$ runs over the irreducible FPF characters of $K$. 
\end{lemma}
\begin{proof}
If $\chi_{\lambda,\tau}$ is FPF, then $\lambda$ is faithful and $\tau$ is FPF, and hence $H_\lambda = K$. It remains to verify the converse direction.

Suppose $\lambda$ is faithful and $\tau$ is FPF. We apply Lemma~\ref{lem:21fpfpbf} to show that $\chi=\chi_{\lambda,\tau}$ is FPF. Write $A=C_m\rtimes K.$ For any element $g\in G$ of prime order $p$, Mackey's formula gives
\[\langle \chi|_{\langle g\rangle},1 \rangle=\frac{1}{p}\sum_{k=0}^{p-1}\chi(g^k)=\frac{1}{p}\sum_{k=0}^{p-1}\sum_{\substack{s\in H/K \\ s^{-1}g^ks\in A}}(\lambda\boxtimes\tau)(s^{-1}g^ks).\]
Since $(m,|H|)=1$, we have two cases: if $p\mid m$ then $g\in C_m$; otherwise, $g$ is conjugate to an element of $H$. 
\begin{itemize}
\item If $g\in C_m,$ then $\langle g\rangle\leq C_m\triangleleft A.$ Therefore
\[\langle \chi|_{\langle g\rangle},1 \rangle=\frac{1}{p}\sum_{s\in H/K}\sum_{k=0}^{p-1}(s\cdot\lambda)(g^k)\cdot \tau(1).\]
As $\lambda$ is faithful, every $s\cdot\lambda$ is also faithful. By Lemma~\ref{lem:21fpfpbf},
$\sum_{k=0}^{p-1}(s\cdot\lambda)(g^k)=0,$
which implies $\langle \chi|_{\langle g\rangle},1 \rangle=0.$

\item If $g\in x^{-1}Hx$ for some $x\in G,$ we may assume $g\in H$. Lemma~\ref{lem:3.1primact} yields $g\in K$. Since $K\triangleleft  H,$ we obtain
\[\langle \chi|_{\langle g\rangle},1 \rangle=\frac{1}{p}\sum_{k=0}^{p-1}\sum_{s\in H/K}\tau(s^{-1}g^ks).\]
Because $\tau$ is FPF, we have $\sum_{k=0}^{p-1}\tau(s^{-1}g^ks)=0,$ hence $\langle \chi|_{\langle g\rangle},1 \rangle=0.$ This completes the proof.
\end{itemize}
\end{proof}

\begin{theorem}\label{thm:fpfweiyi}
Let $G$ be an FPF group. Then the irreducible FPF characters of $G$ form a single orbit under the action of the Galois group. Equivalently, $G$ admits a unique irreducible rational FPF representation up to isomorphism.
\end{theorem}

Before proving Theorem~\ref{thm:fpfweiyi}, we treat two basic examples. Let $\varphi$ be Euler’s totient function, and set $\zeta_m=e^{2\pi i/m}.$

\begin{example}\label{thmcycphi31}
The simplest example is the cyclic group $C_m=\langle g\rangle.$ The complex irreducible representations of $C_m$ are given by $\rho_k(g) = \zeta_m^k$ for $1\le k \le m.$ Each $\rho_k$ has Schur index $1$ over $\mathbb{Q}.$ Then $\rho_k$ is FPF if and only if $(k,m) = 1.$ The $\varphi(m)$ complex irreducible FPF representations are mutually Galois conjugate, hence they assemble into the unique irreducible rational FPF representation $\mathbb{Q}(\zeta_m),$ of dimension $\varphi(m).$ Consequently, the dimension of any rational FPF representation of $C_m$ is divisible by $\varphi(m).$ 
\end{example}

\begin{example}\label{exa:q4mfpf}
Consider another example $Q_{4m}$. By Example~\ref{exam25type1}, every irreducible FPF representation of $Q_{4m}$ is given by
\[\rho_k(a)=
\begin{pmatrix}
\zeta_{2m}^k & 0\\  0 & \zeta_{2m}^{-k} \end{pmatrix},
\qquad  \rho_k(b)=\begin{pmatrix} 0 & 1\\ (-1)^k & 0 \end{pmatrix},\]
where $(k,2m)=1.$ These representations are mutually Galois conjugate.
\end{example}

\begin{proof}[Proof of Theorem \ref{thm:fpfweiyi}]
We first treat Types~\Romannum{1}--\Romannum{5} by induction on $G$; Type~\Romannum{6} is handled at the end.

\paragraph{(a) The induction step.}
Let $G$ be of Types \Romannum{1}--\Romannum{5}, with decomposition $G=C_m\rtimes_\alpha H$ and $K=\ker\alpha$. Suppose the irreducible FPF characters of $K$ form a single Galois orbit, we show the same holds for $G$.

Let $\lambda$ be a faithful character of $C_m$, and $\tau$ an irreducible FPF character of $K$. The values of $\lambda$ lie in $\mathbb{Q}(\zeta_{m})$, and the values of $\tau$ lie in $\mathbb{Q}(\zeta_{|K|})$. Since $(m,|K|)=1$, the two cyclotomic fields are linearly disjoint:
\[\mathbb{Q}(\zeta_{m})\cap\mathbb{Q}(\zeta_{|K|})=\mathbb{Q},\qquad
\operatorname{Gal}\big(\mathbb{Q}(\zeta_{m},\zeta_{|K|})/\mathbb{Q}\big)\cong(\mathbb{Z}/m\mathbb{Z})^{\times}\times(\mathbb{Z}/|K|\mathbb{Z})^{\times}.\]
By Lemma~\ref{lem:26fpfsemp}, the Galois orbit of $\lambda\boxtimes\tau$ is the product of the Galois orbits of $\lambda$ and $\tau$. Combined with Example~\ref{thmcycphi31}, induction shows that all irreducible FPF characters of $G$ lie in a single Galois orbit.

\paragraph{(b) The case $m=1$.}
If $m>1$, then $K\le H=G/C_m$ has order $|G|/m<|G|$, and $K$ is FPF. By induction, the FPF characters of $K$ form a single orbit, so by (a) the same holds for $G$. It remains to handle $m=1$. Now consider the Types \Romannum{1}--\Romannum{5} groups with no nontrivial normal cyclic Hall subgroup of odd order, which are exactly:
\begin{enumerate}[label=(\roman*)]
\item  cyclic $2$-groups $C_{2^a}$ (Type~\Romannum{1});
\item  generalized quaternion groups $Q_{2^a}$, $a\ge 3$ (Type~\Romannum{2});
\item  the binary tetrahedral group $\operatorname{SL}_2(\mathbb{F}_3)$ and its central extensions by cyclic $3$-groups (Type~\Romannum{3} with $n=3^s$);
\item  the binary octahedral group $2O$ and its central extensions by cyclic $3$-groups (Type~\Romannum{4} with $n=3^s$);
\item  the binary icosahedral group $\operatorname{SL}_2(\mathbb{F}_5)$ (Type~\Romannum{5} with $n=1$).
\end{enumerate}
Cases (\romannumeral1) and (\romannumeral2) are Example~\ref{thmcycphi31} and \ref{exa:q4mfpf}.
For $\operatorname{SL}_2(\mathbb{F}_3)$, $2O$, $\operatorname{SL}_2(\mathbb{F}_5)$ (finite subgroups of $\operatorname{SU}(2)$), we apply Lemma~\ref{lem:21fpfpbf} together with the character tables of these groups to identify all irreducible FPF representations, which turn out to be exactly the $2$-dimensional spin representations with $-1\mapsto -I_2$:

\begin{center}\begin{tabular}{lcc}
\toprule
Group &  Irreducible FPF representations & Character field \\
\midrule
$\operatorname{SL}_2(\mathbb{F}_3)$   & one $2$-dim & $\mathbb{Q}$ \\
$2O$    & two $2$-dim & $\mathbb{Q}(\sqrt{2})$ \\
$\operatorname{SL}_2(\mathbb{F}_5)$    & two $2$-dim & $\mathbb{Q}(\sqrt{5})$ \\
\bottomrule
\end{tabular}\end{center}
In the last two cases the two characters are exchanged by $\sqrt2\mapsto-\sqrt2$ and $\sqrt5\mapsto-\sqrt5$. Hence each of these three groups has a single Galois orbit of FPF characters.

A central extension by $C_{3^{s-1}}$ is obtained by tensoring a spin character with a faithful character of the central $C_{3^{s-1}}$. This replaces the character field by $\mathbb{Q}(\zeta_{3^{s-1}})$ in case (\romannumeral3) and by $\mathbb{Q}(\sqrt{2},\zeta_{3^{s-1}})$ in case (\romannumeral4); Since the cyclotomic factor is linearly disjoint from $\mathbb{Q}(\sqrt{2})$ and the Galois group acts transitively on it, the orbit is still unique.

\paragraph{(c) Type~\Romannum{6}.} Let $G=\langle G_1,S\rangle$ with $G_1=C_n\times\operatorname{SL}_2(\mathbb{F}_5)$, $[G:G_1]=2$, $S^2=-I$, and $S$ inducing the outer automorphism $\theta$ of $\operatorname{SL}_2(\mathbb{F}_5)$, $\theta$ swaps the two spin characters $\chi,\chi'$. The irreducible FPF characters of $G_1$ are $\lambda\boxtimes\chi$ and $\lambda\boxtimes\chi'$, $\lambda$ faithful on $C_n$. Since $S$ exchanges $\chi$ and $\chi'$, the irreducible FPF characters of $G$ are $\operatorname{Ind}_{G_1}^{G}\bigl(\lambda\boxtimes\chi\bigr)$ with $\lambda$ faithful on $C_n$ (where $\lambda$ and $\lambda\circ s$, $s$ the automorphism of $C_n$ induced by $S$, give the same
character). The induced character contains both $\chi$ and $\chi'$, hence is fixed by $\sqrt5\mapsto-\sqrt5$. The Galois action therefore reduces to the transitive action on $\lambda$, and the FPF characters of $G$ form a single orbit. For $n=1$, $G$ has order $240$ and a single irreducible FPF character of degree $4$.

This proves the first assertion. The final equivalence follows from Lemma~\ref{lemgalcorr}: distinct Galois orbits of irreducible complex characters give rise to non-isomorphic simple rational modules, so a single orbit corresponds to a unique irreducible rational FPF representation.
\end{proof}

We already know that the irreducible rational FPF representation is unique. However, going from rational representations to integral ones is more complicated. According to the Jordan-Zassenhaus finiteness theorem \cite{curtis1966rep, Reiner1970as}, each irreducible rational representation corresponds to only finite integral representations, yet this correspondence is far from bijective. Concretely, for an irreducible $\mathbb{Q}[G]$-module $V,$ the set of integral forms
\[\mathcal{O}(V) = \{\mathbb{Z}[G]\text{-lattice } L \mid L\otimes_{\mathbb{Z}}\mathbb{Q} \cong V\}\]
contains finitely many $\mathbb{Z}$-equivalence classes. The size of this set depends on deep arithmetic invariants. As a concrete example, when $G=C_m$ is cyclic and $V=\mathbb{Q}(\zeta_m),$ Taussky-Todd \cite{TausskyTodd40m} show that the $\mathbb{Z}$-equivalence classes $\mathcal{O}(\mathbb{Q}(\zeta_m))$ of irreducible integral FPF representations of $C_m$ are in bijection with the ideal class group of the cyclotomic ring $\mathbb{Z}[\zeta_m]$; see also Reiner \cite{Reiner1970as}.

\begin{proposition}\label{corozhcdimn}
Let $V$ be a rational FPF representation of $G.$ Then for every Sylow $p$-subgroup $P$ of $G,$ $\varphi(|P|)$ divides $\dim_{\mathbb{Q}}V.$
\end{proposition}
\begin{proof}
Write $n=\dim_{\mathbb{Q}}V$ and $\rho \colon G\to \operatorname{GL}(V).$ By Theorem \ref{thm22pqsylow}, every odd Sylow $p$-subgroup $P\leq G$ is cyclic. The restriction $\rho|_{P}$ is a rational FPF representation of $P,$ so Example \ref{thmcycphi31} yields $\varphi(|P|)\mid n.$ If the Sylow $2$-subgroup $P$ is cyclic, the same argument gives $\varphi(|P|)\mid n.$ If the Sylow $2$-subgroup is $Q_{2^{a}}$ with $a\ge 3,$ every irreducible complex FPF representation of $Q_{2^{a}}$ has dimension $2$ and is of the form $\rho_k$ with $(k,2^{a-1})=1.$ Their direct sum $\oplus_{(k,2^{a-1})=1}\rho_k$ has dimension $2^{a-1},$ which forces $2^{a-1}\mid n.$
\end{proof}

\begin{proposition}\label{coro28gsln}
Let $\rho\colon G \to\operatorname{Aut}(L)$ be an integral FPF representation. If $|G|>2,$ then $\rho(G) <\operatorname{SL}(L).$
\end{proposition}
\begin{proof}
For every $g\in G$ of order $m>1,$ combining $\Phi_m(0)=1$ and Lemma~\ref{lem33detai} gives $\det(\rho(g))=(-1)^{\operatorname{rank}L}.$ Since $|G|>2,$ Proposition~\ref{corozhcdimn} implies that $\operatorname{rank}L$ is even, so $\det(\rho(g))=1.$
\end{proof}

\section{FPF $G$-actions on tori}\label{sec:acts}
We recall the standard identification $\operatorname{Hom}(\mathbb{Z}^n,\mathbb{C}^*)\cong (\mathbb{C}^*)^n$ given by evaluating a character at the standard basis vectors $\chi\mapsto (\chi(e_1),\dots,\chi(e_n)).$ Under the exponential map $t\mapsto\exp(2\pi i t),$ this is further identified with $(\mathbb{C}/\mathbb{Z})^n.$ If $\mathbb{Z}^n$ is an integral representation of $G,$ then $\operatorname{Hom}(\mathbb{Z}^n,\mathbb{C}^*)$ admits a natural $G$-action that preserves its algebraic group structure,
\[(g \cdot \chi)(v) = \chi(g^{-1}v).\]

For our purposes, it is convenient to work with the torus $V_{\mathbb{C}}/L$ when studying the induced group action.
\begin{definition}
Let $V$ be a rational FPF representation of $G$, and let $L\subset V$ be a $G$-invariant full-rank lattice. The $G$-action on $V$ extends $\mathbb{C}$-linearly to $V_{\mathbb{C}},$ and it induces an FPF $G$-action on the torus $V_{\mathbb{C}}/L$ given by $g\cdot [x] = [g x].$
\end{definition}

\begin{remark}
Let $\rho$ be an irreducible integral FPF representation of $G.$ Its contragredient representation $\rho^*,$ defined by $\rho^*(g) = \rho(g)^{-\top},$ is also an irreducible integral FPF representation. While $\rho$ and $\rho^*$ are isomorphic as $\mathbb{Q}[G]$-modules since they share the same character, they need not be isomorphic as $\mathbb{Z}[G]$-lattices.
\end{remark}

Given an FPF $G$-action on $V_{\mathbb{C}}/L$ and $g\ne e,$ the fixed-point equation $g\cdot [x] = [x]$ is equivalent to $gx-x\in L,$ which yields
\begin{equation*}
\operatorname{Fix}(g)=\{x\in V_{\mathbb{C}} \mid  gx-x\in L\}\big/L\cong L/(g-\mathrm{id})L.
\end{equation*}
In particular, every element of $\operatorname{Fix}(g)$ is rational, i.e., lies in $V/L.$ The FPF condition guarantees that $\det(g-\mathrm{id})\neq 0$. Moreover, $\operatorname{Fix}(g)$ is a finite subgroup of $V_{\mathbb{C}}/L$ with order $|\operatorname{Fix}(g)|=|\det(g-\mathrm{id})|$. Consequently, only finitely many points admit nontrivial stabilizers, and the quotient space $(V_{\mathbb{C}}/L)/G\cong V_{\mathbb{C}}/(G\ltimes L)$ has finitely many singular points. If $m=\operatorname{ord}(g)>1$ and $\operatorname{rank}L=n,$ it follows from Lemma \ref{lem33detai} that
\begin{equation}\label{eqfix31det}
|\operatorname{Fix}(g)|=|\det(g-\mathrm{id})|=\Phi_m(1)^{\frac{n}{\varphi(m)}}.
\end{equation}
The value of $\Phi_m(1)$ is a standard fact about cyclotomic polynomials
\[\Phi_m(1) =
\begin{cases}
0 & \text{if } m = 1,\\
p & \text{if } m = p^k \text{ is a prime power},\\
1 & \text{otherwise}.
\end{cases}\]

Let $F$ denote the set of non-free points of $V_{\mathbb{C}}/L$ with nontrivial stabilizer, i.e.,
\[F=\{x\in V_{\mathbb{C}}/L\mid G_x\ne\{e\}\}=\bigcup_{g\ne e}\operatorname{Fix}(g),\]
whose points all lie in $V/L.$ We use M\"{o}bius functions \cite{rota1964} to count the points of $F.$ For the M\"{o}bius function of the subgroup lattice of a finite group, see also \cite{hawkes1989}.

Let $\operatorname{Sub}(G)$ denote the set of all subgroups of $G,$ partially ordered by inclusion. The Möbius function $\mu_G: \operatorname{Sub}(G) \times \operatorname{Sub}(G) \to \mathbb{Z}$ is the unique function defined recursively by:
\begin{enumerate}
\item $\mu_G(K, H) = 0$ if $K \nleq H$,
\item $\mu_G(K, H) = 1$ if $K = H$,
\item $\mu_G(K, H) = -\sum_{K \leq L < H} \mu_G(K, L)$ if $K < H.$
\end{enumerate}

Cyclic group $C_m=\langle g\rangle$ is a basic example. For each $d\mid m,$ there exists a unique subgroup $H_d=\langle g^{m/d}\rangle$ of order $d,$ and $H_{d_1}\le H_{d_2}$ if and only if $d_1\mid d_2.$ Hence the subgroup lattice of $C_m$ is order-isomorphic to the divisor lattice of $m,$ so the lattice-theoretic M\"{o}bius function coincides with the classical M\"{o}bius function, $\mu_{C_m}(\{e\},H_d)=\mu(d).$

 For any subgroup $H\leq G,$ define 
\[N(H):=\#\{x\in V_{\mathbb{C}}/L \mid  G_x=H\},\quad \operatorname{Fix}(H)=\bigcap_{h\in H}\operatorname{Fix}(h)\cong L/\sum_{h\in H}(h-\mathrm{id})L.\] 
We obtain the lattice decomposition and Möbius inversion formula
\[|\operatorname{Fix}(H)|= \sum_{H\leq K} N(K),\quad N(H)= \sum_{H\leq K} \mu_G(H,K)|\operatorname{Fix}(K)|.\]
Let $f$ be a function on $\operatorname{Sub}(G).$ Summing over all non-free points, we have
\begin{equation}\label{cardfsing32}
\sum_{x\in F} f(G_x) = \sum_{\{e\}<H\le K}\mu_G(H,K)f(H)|\operatorname{Fix}(K)|.
\end{equation}
Setting $f \equiv 1$ gives 
\begin{equation}\label{cardfsing33}
|F| = -\sum_{\{e\}<H\leq G} \mu_G\bigl(\{e\}, H\bigr)\cdot  |\operatorname{Fix}(H)|.
\end{equation}
Substituting different functions $f$ into \eqref{cardfsing32} yields various geometric quantities associated with the non-free points.

Recall that every FPF group satisfies the $pq$-condition. The following lemmas establish properties for groups satisfying this condition.

\begin{lemma}\label{lem33pqcin}
Let $H$ be a finite group satisfying the $pq$-condition. If $|H|$ is even, then $H$ has a unique involution (hence central).
\end{lemma}
\begin{proof}
Suppose $H$ has two distinct involutions $a\ne b.$ Then $\langle a,b\rangle\cong D_{2m}$ with $m=\operatorname{ord}(ab).$ If $m$ is even, then $\langle a,(ab)^{m/2}\rangle\cong C_2\times C_2$ is a non-cyclic subgroup of order $4=pq,$ a contradiction. If $m$ is odd, choose an odd prime $p\mid m.$ Then $D_{2m}$ has a subgroup $D_{2p},$ a non-cyclic group of order $2p,$ contradiction.
\end{proof}

\begin{lemma}\label{lemmuvsh34}
Let $H$ be a finite group satisfying the $pq$-condition. Then for every non-cyclic subgroup $K\le H,$ we have $\mu_H(\{e\},K)=0.$
\end{lemma}
\begin{proof}
We prove by induction on $|H|.$ If $H$ is cyclic, there is nothing to show. Assume $H$ non-cyclic and $\mu_H(\{e\},K)=0$ for every proper non-cyclic $K<H.$ Hence, by the recursive definition of $\mu_H,$
\[\mu_H(\{e\},H)=-\sum_{K<H}\mu_H(\{e\},K)
= -\sum_{\substack{K<H\\ K\text{ cyclic}}}\mu(|K|).\]
Let $c_d$ be the number of cyclic subgroups of $H$ of order $d.$ Since $H$ is non-cyclic, $c_{|H|}=0.$ Only square-free $d$ contribute, so
\[\mu_H(\{e\},H)=-S(H),\qquad S(H):=\sum_{\substack{d\mid |H|\\ d\text{ square-free}}}\mu(d)c_d.\]
We split into two cases to show $S(H) = 0.$

\textbf{Case 1: $|H|$ even.} By Lemma \ref{lem33pqcin}, $H$ possesses a unique involution $z.$ For each odd square-free $d\mid |H|,$ the map 
\[C\mapsto C\langle z\rangle\cong C_d\times C_2\cong C_{2d}\]
 gives a bijection between cyclic subgroups of order $d$ and those of order $2d.$ Thus $c_{2d}=c_d.$ Since $c_1=c_2=1$ and $\mu(d)+\mu(2d)=0,$ we obtain
 \[S(H)=\mu(1)c_1+\mu(2)c_2+\sum_{\substack{d>1\\ d\text{ odd square-free}}}\mu(d)(c_d-c_{2d})=0.\]

\textbf{Case 2: $|H|$ odd.} By the $pq$-condition, every Sylow $p$-subgroup is cyclic; otherwise $C_p\times C_p$ occurs. So by Theorem 5.4.1 of \cite{wolf2011spaces}, $H$ is metacyclic: 
\[H=\langle a,b\mid a^m=b^n=1,\ bab^{-1}=a^r\rangle\]
 with $(m,n)=1$ and $m,n$ odd. For each square-free $d\mid mn,$ write $d=d_1d_2$ with $d_1\mid m,$ $d_2\mid n$ and $(d_1,d_2)=1.$ Any cyclic subgroup of order $d$ must intersect $\langle a\rangle$ in the unique subgroup of order $d_1$ and project onto the unique subgroup of order $d_2$ in $H/\langle a\rangle.$ Hence it contains the subgroup $K_1=\langle a^{m/d_1}, b^{n/d_2}\rangle,$ which has order $d$ and is cyclic (otherwise it would contain a non-cyclic subgroup of order $pq,$ contradicting the $pq$-condition). Thus $K_1$ is the unique cyclic subgroup of order $d,$ so $c_d=1.$ Then
\[S(H)=\sum_{\substack{d\mid mn\\ d\text{ square-free}}}\mu(d)=0.\]
In both cases $S(H)=0.$ This completes the induction.
\end{proof}

Although \eqref{cardfsing33} involves contributions from non-cyclic subgroups, Lemma~\ref{lemmuvsh34} implies that only cyclic subgroups need to be considered. This leads to the following theorem.

\begin{theorem}\label{thmnff35}
For an FPF $G$-action on $n$-dimensional torus $V_{\mathbb{C}}/L,$ the number of non-free points is given by 
\[|F| = -\sum_{d\mid |G|} \mu(d)\cdot c_d(G)\cdot\Phi_{d}(1)^{\frac{n}{\varphi(d)}},\]
where $c_d(G)$ is the number of cyclic subgroups of $G$ of order $d.$
\end{theorem}
\begin{proof}
Since $G$ is an FPF group, $G$ satisfies the $pq$-condition. Then the assertion follows directly from \eqref{eqfix31det}, \eqref{cardfsing33} and Lemma~\ref{lemmuvsh34}.
\end{proof}

Theorem~\ref{thmnff35} implies that the number $|F|$ is independent of the choice of the $G$-invariant lattice $L,$ as it depends only on the dimension of the rational FPF representation $V.$ The same phenomenon holds for the singularities and Euler characteristic of the quotient space. 

We now turn to the singular locus of the quotient space. Let $\pi: V_{\mathbb{C}}/L\to V_{\mathbb{C}}/(G\ltimes L)$ be the canonical projection. Since every point with nontrivial stabilizer lies in $F,$ the singular set of $V_{\mathbb{C}}/(G\ltimes L)$ is precisely $\pi(F).$

\begin{theorem}\label{thmpif36}
For an FPF $G$-action on $n$-dimensional torus $V_{\mathbb{C}}/L,$ the number of singular points of the quotient space $V_{\mathbb{C}}/(G\ltimes L)$ is given by 
\[|\pi(F)|=\frac{1}{|G|}\sum_{d\mid |G|} (\varphi(d)-\mu(d))\cdot c_d(G)\cdot\Phi_{d}(1)^{\frac{n}{\varphi(d)}}.\]
\end{theorem}
\begin{proof}
Applying Burnside's lemma to the $G$-action on $F,$ we have
\begin{align*}
|\pi(F)|=&\frac{1}{|G|}\sum_{g\in G}|\operatorname{Fix}(g) \cap F|=\frac{1}{|G|}\Bigl(|F|+\sum_{g\in G}|\operatorname{det}(g-\mathrm{id})|\Bigr)  \\
=& \frac{|F|}{|G|}+\frac{1}{|G|}\sum_{d\mid |G|}c_d(G)\cdot\varphi(d)\cdot\Phi_{d}(1)^{\frac{n}{\varphi(d)}}
\end{align*}
where the last equality holds because there are $c_d(G)\varphi(d)$ elements of order $d.$ Substituting $|F|$ from Theorem \ref{thmnff35} completes the proof.
\end{proof}

Recall the orbifold Euler characteristic of quotient orbifolds from \cite{Sataketgb57}. For a finite group action on a space $X$ with finite CW-structure compatible with the action, the orbifold Euler characteristic is defined by
\[\chi_{\mathrm{orb}}(X/G):=\frac{1}{|G|}\sum_{g\in G}\chi(\operatorname{Fix}(g)).\] 
For an FPF $G$-action on $V_{\mathbb{C}}/L,$ since $\operatorname{Fix}(g)$ is finite for $g\ne e$ and $\chi(V_{\mathbb{C}}/L)=0,$ we obtain
\begin{equation}\label{equchiord9}
\chi_{\mathrm{orb}}(V_{\mathbb{C}}/(G\ltimes L))=|\pi(F)|-\frac{1}{|G|}|F|.
\end{equation}

\begin{corollary}\label{cor:eulerg37}
For an FPF $G$-action on $n$-dimensional torus $V_{\mathbb{C}}/L,$ the orbifold Euler characteristic of the quotient $V_{\mathbb{C}}/(G\ltimes L)$ is given by
\[\chi_{\mathrm{orb}}(V_{\mathbb{C}}/(G\ltimes L))=\frac{1}{|G|}\sum_{d\mid |G|} \varphi(d)\cdot c_d(G)\cdot\Phi_{d}(1)^{\frac{n}{\varphi(d)}}.\]
($\varphi(d)c_d(G)$ is the number of elements of order $d$ in $G$).
\end{corollary}
\begin{proof}
This follows from \eqref{equchiord9} together with Theorems~\ref{thmnff35} and~\ref{thmpif36}.
\end{proof}

\section{Examples: cyclic groups and generalized quaternion groups}\label{sec4:exams}

In Section \ref{sec:acts}, we established singularity-counting formulas for FPF group actions on tori. We now apply our general theory to two concrete examples: cyclic groups and generalized quaternion groups. For these groups, the number $c_d(G)$ can be computed explicitly.

We now restrict our attention to cyclic groups $C_m=\langle g\rangle.$ 
\begin{theorem}\label{thmcardf34}
For an FPF $C_m$-action on $n$-dimensional torus $V_{\mathbb{C}}/L,$ the number of non-free points is given by 
\[|F| =1-\omega(m)+ \sum_{p \mid m} p^{\frac{n}{p-1}},\]
where $p$ ranges over the distinct prime divisors of $m,$ and $\omega(m)$ denotes the number of distinct prime divisors of $m.$
\end{theorem}

\begin{proof}
For each $d\mid m,$ we have $c_d=1.$ By Theorem~\ref{thmnff35}, we obtain
\begin{equation}\label{eq2:cyclmobius}
|F|=-\sum_{d\mid m}\mu(d)\cdot \Phi_d(1)^{\frac{n}{\varphi(d)}},
\end{equation}
Using the properties of $\mu(d),$ a direct simplification of \eqref{eq2:cyclmobius} gives
\begin{align*}
|F|=-\sum_{\substack{d=p_1\cdots p_r\\ d\mid m}}(-1)^r\cdot\Phi_{d}(1)^{\frac{n}{\varphi(d)}}=\sum_{p\mid m}p^{\frac{n}{p-1}}-\sum_{\substack{p_1\cdots p_r\mid m\\ r\geq 2}}(-1)^r=1-\omega(m)+\sum_{p\mid m}p^{\frac{n}{p-1}},
\end{align*}
where $p_1,\dots,p_r$ are distinct prime divisors of $m.$ The final equality holds because
\[\sum_{\substack{p_1\cdots p_r\mid m\\ r\geq 2}} (-1)^r=\sum_{r=2}^{\omega(m)}\binom{\omega(m)}{r}(-1)^r=\omega(m)-1.\]
\end{proof}

\begin{example}\label{exa35cma3cs}
Let $C_m = \langle g \rangle $ act by multiplication $g\cdot x = \zeta_m x$ on $\mathbb{Z}[\zeta_m],$ and the matrix of $g$ under the standard integral basis of $\mathbb{Z}[\zeta_m]$ is exactly the companion matrix of $\Phi_m(x).$ This induces an FPF $C_m$-action on $\mathbb{Q}(\zeta_m)_{\mathbb{C}}/\mathbb{Z}[\zeta_m].$ 
\begin{itemize}
\item If $m=p$ is prime, then $\varphi(p)=p-1$ and $|F|=p.$ In fact, $F=\operatorname{Fix}(g)$ is a cyclic subgroup of order $p,$ generated by $\Bigl( \frac{1}{p}, \frac{2}{p}, \dots, \frac{p-1}{p}\Bigr)\in (\mathbb{Q}/\mathbb{Z})^{p-1}.$
\item If $m=p^k$ is a prime power, then $\varphi(p^k)=p^{k-1}(p-1)$ and $|F|=p^{p^{k-1}}.$ In this case, $F=\operatorname{Fix}(g^{p^{k-1}})$ is an abelian subgroup isomorphic to $(\frac{1}{p}\mathbb{Z}/\mathbb{Z})^{p^{k-1}}.$
\item If $m$ has at least two distinct prime factors, then $F$ is not a subgroup of torus. Indeed, for any $X = (x_1,\dots, x_{\varphi(m)}) \in F,$ each nonzero coordinate $x_i \in \mathbb{Q}/\mathbb{Z}$ is of the form $k/p$ for some $k\in\{1, \ldots, p-1\}$ and some prime divisor $p$ of $m.$ However, the sum of two such elements may have a coordinate with denominator divisible by two distinct primes $p$ and $q,$ which cannot occur in any element of $F.$ Hence $F$ is not closed under addition.
\end{itemize}
\end{example}

\begin{theorem}\label{thm36cardpif}
For an FPF $C_m$-action on $n$-dimensional torus $V_{\mathbb{C}}/L,$ the number of singular points of the quotient space $V_{\mathbb{C}}/(C_m\ltimes L)$ is given by 
\[|\pi(F)|=1-\frac{\omega(m)}{m}+\frac{1}{m}\sum_{p\mid m}p^{\frac{n}{p-1}}+\frac{1}{m}\sum_{p^k\mid m} \varphi(p^k)\Bigl(p^{\frac{n}{\varphi(p^k)}}-1\Bigr),\]
where $p$ ranges over the distinct prime divisors of $m,$ and $p^k$ ranges over all prime power divisors of $m.$
\end{theorem}
\begin{proof}
It follows from Theorem~\ref{thmpif36} that
\begin{align}
|\pi(F)|&=\frac{1}{m}\Bigl(|F|+\sum_{d\mid m} \varphi(d)\cdot\Phi_d(1)^{\frac{n}{\varphi(d)}} \Bigr)  \notag \\
&=\frac{1}{m}\Bigl(|F|+\sum_{p^k\mid m} \varphi(p^k)\cdot p^{\frac{n}{\varphi(p^k)}}+ \sum_{\substack{d\mid m \\ \omega(d)>1}}\varphi(d)\Bigr). \label{pifexpc35}
\end{align}
Since $\sum_{d\mid m}\varphi(d)=m,$ the last sum in \eqref{pifexpc35} becomes
\[\sum_{\substack{d\mid m \\ \omega(d)>1}} \varphi(d)
= m-1-\sum_{p^k\mid m} \varphi(p^k).\]
Substituting this and $|F|$ from Theorem \ref{thmcardf34} completes the proof.
\end{proof}

\begin{example}
Let $m=p$ be prime and $n=p-1.$ By Example \ref{exa35cma3cs} we have $F=\operatorname{Fix}(g)$ with $|F|=p.$ Every nonzero element in $F$ has full stabilizer $C_p,$ so the $C_p$-action on $F$ is trivial, and we immediately get $|\pi(F)|= p.$ Theorem \ref{thm36cardpif} gives the same value
\[|\pi(F)| = 1-\frac{1}{p}+\frac{1}{p}\cdot p+\frac{(p-1)^2}{p}=p.\]
\end{example}

\begin{corollary}
For an FPF $C_m$-action on $V_{\mathbb{C}}/L,$ the orbifold Euler characteristic of $V_{\mathbb{C}}/(C_m\ltimes L)$ is given by
\[\chi_{\mathrm{orb}}(V_{\mathbb{C}}/(C_m\ltimes L))=\frac{m-1}{m}+\frac{1}{m}\sum_{p^k\mid m} \varphi(p^k)\Bigl(p^{\frac{n}{\varphi(p^k)}}-1\Bigr),\]
where the sum ranges over all prime power divisors $p^k$ of $m.$ 
\end{corollary}
\begin{proof}
This follows from \eqref{equchiord9} together with Theorems~\ref{thmcardf34} and~\ref{thm36cardpif}.
\end{proof}

We next study FPF actions of the generalized quaternion group $Q_{4m}.$ Recall that $Q_{4m}$ is defined by the presentation
\[Q_{4m}=\langle a,b\mid a^{2m}=1, b^{2}=a^{m}, bab^{-1}=a^{-1}\rangle.\]
One may easily verify the subgroup classification of $Q_{4m}.$ Every subgroup of $ Q_{4m} $ is of one of the following types:
\begin{itemize}
\item Cyclic subgroups contained in $\langle a\rangle $: $\langle a^{2m/d}\rangle$ of order $d,$ with $d\mid 2m;$
\item Cyclic subgroups not contained in $\langle a\rangle$: $\langle a^kb\rangle$ of order $4,$ with $0\leq k<m;$
\item Non-cyclic subgroups: $\langle a^{m/d},a^kb\rangle$ with $d\mid m,d\ne 1$ and $0\leq k<m/d,$ which are isomorphic to $Q_{4d}.$
\end{itemize}

\begin{theorem}\label{thm:q4mnonf62}
For an FPF $Q_{4m}$-action on $V_{\mathbb{C}}/L,$ the number of non-free points is  
\[|F|=1-\omega(2m)+\sum_{p\mid 2m}p^{\frac{n}{p-1}},\]
where $p$ ranges over the distinct prime divisors of $2m.$
\end{theorem}
\begin{proof}
For each square-free $d\mid 2m,$ we have $c_d=1.$ By Theorem~\ref{thmnff35}, we obtain
\[|F|=-\sum_{d\mid 2m}\mu(d)\cdot\Phi_d(1)^{\frac{n}{\varphi(d)}},\]
which recovers exactly the expression \eqref{eq2:cyclmobius}. Applying the result of Theorem~\ref{thmcardf34}, the assertion follows.
\end{proof}

\begin{theorem}\label{thmq4msgq63}
For an FPF $Q_{4m}$-action on $V_{\mathbb{C}}/L,$ the number of singular points of the quotient space $V_{\mathbb{C}}/(Q_{4m}\ltimes L)$ is given by 
\[|\pi(F)|=2^{\frac{n}{2}-1}+\frac{1}{2}-\frac{\omega(2m)}{4m}+\frac{1}{4m}\sum_{p\mid 2m}p^{\frac{n}{p-1}}+\frac{1}{4m}\sum_{p^k\mid 2m} \varphi(p^k)\left(p^{\frac{n}{\varphi(p^k)}}-1\right),\]
where $p$ ranges over the distinct prime divisors of $2m,$ and $p^k$ ranges over all prime power divisors of $2m.$
\end{theorem}
\begin{proof}
Now, the cyclic subgroups of $Q_{4m}$ are of two types. When $d=4,$
\[\frac{1}{4m}(\varphi(4)-\mu(4))m\Phi_{4}(1)^{\frac{n}{\varphi(4)}}=2^{\frac{n}{2}-1}.\]
By Theorem~\ref{thmpif36},
\begin{align*}
|\pi(F)|=2^{\frac{n}{2}-1}+\frac{1}{4m}\sum_{d\mid 2m} (\varphi(d)-\mu(d))\cdot\Phi_{d}(1)^{\frac{n}{\varphi(d)}}=2^{\frac{n}{2}-1}+\frac{1}{2}s,
\end{align*}
where $s$ is the number of singular points of $V_{\mathbb{C}}/(C_{2m}\ltimes L)$ given by Theorem~\ref{thm36cardpif}. This completes the proof.
\end{proof}

\begin{corollary}
For an FPF $Q_{4m}$-action on $V_{\mathbb{C}}/L,$ the orbifold Euler characteristic of $V_{\mathbb{C}}/(Q_{4m}\ltimes L)$ is given by
\[\chi_{\mathrm{orb}}(V_{\mathbb{C}}/(Q_{4m}\ltimes L))=2^{\frac{n}{2}-1}+\frac{1}{2}-\frac{1}{4m}+\frac{1}{4m}\sum_{p^k\mid 2m} \varphi(p^k)\left(p^{\frac{n}{\varphi(p^k)}}-1\right),\]
where $p^k$ ranges over all prime power divisors of $2m.$
\end{corollary}
\begin{proof}
This follows from \eqref{equchiord9} together with Theorems~\ref{thm:q4mnonf62} and~\ref{thmq4msgq63}.
\end{proof}

\section{Cyclic subgroup counts for FPF groups}\label{sec5:cycfpf}
In this section, we count $c_d(G),$ the number of cyclic subgroups of order $d$ of an FPF group $G.$ This is equivalent to counting $a_d(G),$ the number of elements of order $d,$ since $a_d(G)=\varphi(d)c_d(G).$

\subsection{Type \Romannum{1}-\Romannum{5}}

\begin{lemma}\label{lem:ordsh51}
Let $G=C_m\rtimes_\alpha H$ with $(m,|H|)=1.$ Write its elements as $(i,h)$ with multiplication
$(i,h_1)(j,h_2)=(i+\alpha(h_1)j,h_1h_2).$ For $h\in H,$ define
\[S_h:= \sum_{t=0}^{\operatorname{ord}(h)-1}\alpha(h)^t \pmod m.\]
Then for any $(i,h)\in G,$
\[\operatorname{ord}((i,h))=\frac{m\cdot\operatorname{ord}(h)}{(m,iS_h)}.\]
\end{lemma}

\begin{proof}
Induction gives $(i,h)^k=(i\sum_{t=0}^{k-1}\alpha(h)^t,h^k).$ Setting $k=\operatorname{ord}(h),$ we obtain
\[(i,h)^{\operatorname{ord}(h)}=(iS_h,e)\in C_m.\]
The order of $iS_h$ in $C_m$ equals $\frac{m}{(m,iS_h)}.$ This proves the claim.
\end{proof}

From this order formula we immediately deduce the following theorem.

\begin{theorem}\label{thm:adgfpf52}
Let $G=C_m\rtimes_\alpha H$ with $(m,|H|)=1.$ For any $\mid |G|,$ write uniquely $d = d_1 d_2,$ where $d_1 \mid m,d_2 \mid |H|.$ Then
\begin{equation}
a_{d_1d_2}(G)=\varphi(d_1)\sum_{\substack{h\in H\\ \operatorname{ord}(h)=d_2}}
(m,S_h) \left[d_1\mid \frac{m}{(m,S_h)}\right],
\end{equation}
where $[\cdot]$ denotes the indicator function, and $S_h$ is defined in Lemma~\ref{lem:ordsh51}.
\end{theorem}

Let $G=K\times \operatorname{SL}_2(\mathbb{F}_5)$ be of Type \Romannum{5}, where $K$ is of Type \Romannum{1} with $(|K|,30)=1.$ The order distribution of $\operatorname{SL}_2(\mathbb{F}_5)$ is given in the table below
\begin{table}[H]
  \centering
  \caption{Order distribution of $\operatorname{SL}_2(\mathbb{F}_5)$}
  \label{tab:sl2f5_order}
\begin{tabular}{|c|ccccccc|}
\hline 
$d$ & $1$ & $2$ & $3$ & $4$ & $5$ & $6$ & $10$ \\
\hline
$a_d$ & $1$ & $1$ & $20$ & $30$ & $24$ & $20$ & $24$ \\
\hline
\end{tabular}\end{table}
\noindent For each $d\mid 120|K|,$ write uniquely $d = d_1d_2$ with $d_1\mid |K|$ and $d_2\mid 120.$ Then 
\[a_{d_1d_2}(G)=\sum_{\substack{d_1\mid |K| \\d_2\mid 120}}a_{d_1}(K)\cdot a_{d_2}(\operatorname{SL}_2(\mathbb{F}_5)),\]
where $a_{d_1}(K)$ is given in Theorem~\ref{thm:adgfpf52}.

\subsection{Type \Romannum{6}}
Let $G=K\rtimes H$ be of Type \Romannum{6}, where $K$ is of Type \Romannum{1}, $H=\langle\operatorname{SL}_2(\mathbb{F}_5),S\rangle,$ and $(|K|,30)=1.$ We also write $G_1=K\times\operatorname{SL}_2(\mathbb{F}_5)$. The order distribution of $H$ is given in the table below

\begin{table}[H]
  \centering
  \caption{Order distribution of $\langle\operatorname{SL}_2(\mathbb{F}_5),S\rangle$}
  \label{tab:h_order}
\begin{tabular}{|c|ccccccccc|}
\hline 
$d$ &1&2&3&4&5&6&8&10&12  \\ 
\hline   
$a_d$ &1&1&20&50&24&20&60&24&40  \\
\hline
\end{tabular}\end{table}
\noindent Unlike Types \Romannum{1}-\Romannum{5}, the count for Type \Romannum{6} involves the automorphism of $K$ induced by $S$; nevertheless it admits a completely explicit formula, which we now derive. Let $K=C_m\rtimes C_n=\langle A,B\rangle$ with $A^m=B^n=1$ and $BAB^{-1}=A^r$, and set $\sigma=\operatorname{conj}_K(S)\in\operatorname{Aut}(K)$. Since $S^2=-I\in\operatorname{SL}_2(\mathbb{F}_5)$ and $\operatorname{SL}_2(\mathbb{F}_5)$ acts trivially on $K$, we have $\sigma^2=\mathrm{id}$. Replacing $S$ by a suitable element of the coset $G_1S$ if necessary, we may assume (cf.~\cite[Chapter~6]{wolf2011spaces})
\begin{equation}\label{equ:sigmaab}
\sigma(A)=A^s,\qquad \sigma(B)=B^t,
\end{equation}
where $s^2\equiv 1\pmod m$, $t^2\equiv 1\pmod n$, and $r^{t-1}\equiv 1\pmod m$. The subgroup $\operatorname{SL}_2(\mathbb{F}_5)\le H$ acts trivially on $K$, while $S$ acts through $\sigma$. Hence $G=G_1\sqcup G_1S$.

For each $d\mid 240|K|$ with unique decomposition $d = d_1d_2,$ where $d_1\mid |K|$ and $d_2\mid 120,$ the elements of $G_1$ contribute the Type \Romannum{5} count
\[a_{d_1d_2}(G_1)=a_{d_1}(K) a_{d_2}(\operatorname{SL}_2(\mathbb{F}_5)).\]
It remains to count the elements of the coset $G_1S$. Write $g=k xS$ with $k\in K$, $x\in\operatorname{SL}_2(\mathbb{F}_5)$, and set $h=xS\in H$. Then $h$ lies in the non-trivial coset $\operatorname{SL}_2(\mathbb{F}_5)\cdot S\subset H$.

\begin{lemma}\label{lem:coseto}
For any $h\in\operatorname{SL}_2(\mathbb{F}_5)\cdot S$, we have $\operatorname{ord}(h)\in\{4,8,12\}$.
\end{lemma}
\begin{proof}
Write $h=xS$ with $x\in\operatorname{SL}_2(\mathbb{F}_5).$ Using $S^2=-I$ and $SAS^{-1}=\theta(A)$, we have
\[h^2=xSxS=x(SxS^{-1})S^2=-x\theta(x),\]
where $\theta=\operatorname{Ad}_S$ is the involution on $\operatorname{SL}_2(\mathbb{F}_5)$ introduced in Theorem \ref{thm:FPF_classifi}. So $h^2\in\operatorname{SL}_2(\mathbb{F}_5).$ If $\operatorname{ord}(h)$ is odd, say $2k+1$, then $h=(h^2)^{k+1}\in\operatorname{SL}_2(\mathbb{F}_5)$, contradicting $h\notin\operatorname{SL}_2(\mathbb{F}_5)$. Hence $\operatorname{ord}(h)$ is even.

We next find $\operatorname{ord}(h^2).$ Put $u:=x\theta(x)$. Write $x=\begin{pmatrix}a&b\\c&d\end{pmatrix}\in\operatorname{SL}_2(\mathbb{F}_5)$, one has
\begin{equation}\label{eq:u}
u=\begin{pmatrix}a&b\\c&d\end{pmatrix}\begin{pmatrix}0&-1\\2&0\end{pmatrix}
 \begin{pmatrix}a&b\\c&d\end{pmatrix}
 \begin{pmatrix}0&3\\4&0\end{pmatrix}
=\begin{pmatrix}ad+3b^2 & 2ac+ab \\[2pt] cd+3bd & 2c^2+ad\end{pmatrix}.
\end{equation}
Using $ad-bc=1$,
\begin{equation}\label{eq:trace}
\operatorname{tr}(u)=2ad+3b^2+2c^2=2+3(b+2c)^2.
\end{equation}
Because the squares in $\mathbb{F}_5$ are $\{0,1,4\}$, we see that $\operatorname{tr}(u)\in\{2,0,4\}.$

We may read off $\operatorname{ord}(u)$ from $\operatorname{tr}(u)$. Since $\det(u)=1,$ the Cayley--Hamilton theorem gives $u^2-\operatorname{tr}(u)u+I=0$. We claim that $\operatorname{ord}(u)\in\{1,3,4\}$:
\begin{itemize}
\item If $\operatorname{tr}(u)=2$, then by \eqref{eq:trace} we have $b+2c=0$. From \eqref{eq:u}, $u=(ad+2c^2)I$ is scalar. Then $\det(u)=1$ forces $u=\pm I$, and $\operatorname{tr}(u)=2$ forces $u=I$. Hence $\operatorname{ord}(u)=1$.
\item If $\operatorname{tr}(u)=0$, then $u^2+I=0$. Thus $u^4=I$ and $\operatorname{ord}(u)=4$.
\item If $\operatorname{tr}(u)=4$, then $u^2+u+I=0$. Multiplying by $u-I$
gives $u^3-I=0$, so $\operatorname{ord}(u)=3$.
\end{itemize}

Since $\operatorname{ord}(-u)=2\operatorname{ord}(u)$ when $\operatorname{ord}(u)$ is odd and
$\operatorname{ord}(-u)=\operatorname{ord}(u)$ when $\operatorname{ord}(u)$ is even, we have
\[\operatorname{ord}(h^2)=\operatorname{ord}(-u)=
\begin{cases}
2, & \operatorname{ord}(u)=1,\\
6, & \operatorname{ord}(u)=3,\\
4, & \operatorname{ord}(u)=4,
\end{cases}\]
and hence $\operatorname{ord}(h)=2\operatorname{ord}(h^2)\in\{4,8,12\}.$
\end{proof}

Comparing Table \ref{tab:sl2f5_order} and Table \ref{tab:h_order}, the order distribution of the coset $\operatorname{SL}_2(\mathbb{F}_5)\cdot S$ is

\begin{table}[H]
  \centering
  \caption{Order distribution of $\operatorname{SL}_2(\mathbb{F}_5)\cdot S$}
  \label{tab:coset_order}
\begin{tabular}{|c|ccc|}
\hline
$d$ & $4$ & $8$ & $12$ \\
\hline
$a_{d}(\operatorname{SL}_2(\mathbb{F}_5)\cdot S)$ & $20$ & $60$ & $40$\\
\hline
\end{tabular}\end{table}

Let $g=kh$ with $k\in K$ and $h\in\operatorname{SL}_2(\mathbb{F}_5)\cdot S$. A straightforward induction yields
\[g^r=k\sigma(k)\sigma^2(k)\cdots\sigma^{r-1}(k)h^r.\]
Since $\sigma^2=\operatorname{id}$ and $\operatorname{ord}(h)$ is even, we have 
\begin{equation}\label{lem:cosetorder}
\operatorname{ord}(g)=\operatorname{ord}(h)\cdot\operatorname{ord}(k\sigma(k)).
\end{equation}
For each $d\mid |K|,$ define 
\[b_{d}(\sigma)=\#\{k\in K\mid \operatorname{ord}(k\sigma(k))=d\}.\]
We now compute $b_d(\sigma)$ explicitly.

\begin{lemma}\label{lem:bdsigma}
Write $k=A^iB^j$ with $0\leq i<m$ and $0\leq j<n$, and let $v_j\equiv j(1+t)\pmod n$. Put
\[e_j=\frac{n}{(n,\,j(1+t))},\qquad u_j=1+sr^j,\qquad S_j=\sum_{q=0}^{e_j-1}r^{v_jq}\bmod m,\qquad \tau_j=(m,\,u_jS_j).\]
(Thus $e_j$ is the order of $B^{v_j}$, and $S_j=S_{B^{v_j}}$ in the notation of Lemma~\ref{lem:ordsh51}.) Then the number of $i\in\mathbb{Z}/m\mathbb{Z}$ for which $\operatorname{ord}(A^{iu_j}B^{v_j})=d$ equals
\[\tau_j\,\varphi\Bigl(\frac{d}{e_j}\Bigr)\quad\text{if } e_j\mid d \text{ and } \frac{d}{e_j}\mid\frac{m}{\tau_j},\]
and equals $0$ otherwise. Consequently
\begin{equation}\label{equ:bdsigma}
b_d(\sigma)=\sum_{j=0}^{n-1}\frac{m}{\tau_j}\,\varphi\Bigl(\frac{d}{e_j}\Bigr)\Bigl[\,e_j\mid d\ \text{ and }\ \frac{d}{e_j}\mid\frac{m}{\tau_j}\,\Bigr],
\end{equation}
where $[\cdot]$ denotes the indicator function.
\end{lemma}
\begin{proof}
By \eqref{equ:sigmaab} we have $\sigma(k)=A^{is}B^{jt}$, so
\[k\sigma(k)=A^iB^jA^{is}B^{jt}=A^{i(1+sr^j)}B^{j(1+t)}=A^{iu_j}B^{v_j},\]
where in the second step we used $B^jA^{is}=A^{isr^j}B^j$, which follows from $BAB^{-1}=A^r.$ Applying Lemma~\ref{lem:ordsh51} to $K=C_m\rtimes_\alpha C_n$ with $h=B^{v_j}$, whose order is $e_j$, we obtain
\[\operatorname{ord}(A^{iu_j}B^{v_j})=\frac{m\,e_j}{(m,\,iu_jS_j)}.\]
Write $u_jS_j\equiv\tau_j\tilde w_j\pmod m$ with $(\tilde w_j,\,m/\tau_j)=1$, where $\tau_j=(m,u_jS_j)$. Then $(m,iu_jS_j)=\tau_j\,(m/\tau_j,\,i)$, so with $M_j=m/\tau_j$,
\[\operatorname{ord}(A^{iu_j}B^{v_j})=\frac{M_je_j}{(M_j,\,i)}.\]
This equals $d$ if and only if $(M_j,i)=c$, where $c=M_je_j/d$; such an integer $c$ exists with $c\mid M_j$ precisely when $e_j\mid d$ and $d\mid M_je_j$, that is, when $e_j\mid d$ and $d/e_j\mid M_j$. Finally, for $c\mid M_j$ the number of $i\in\mathbb{Z}/m\mathbb{Z}$ with $(M_j,i)=c$ equals $\tau_j\varphi(M_j/c)$: indeed $i=ci'$ with $(i',M_j/c)=1$, and $i'$ runs over $\mathbb{Z}/(\tau_jM_j/c)\mathbb{Z}$, in which every residue class modulo $M_j/c$ occurs exactly $\tau_j$ times. Since $M_j/c=d/e_j$, the assertion follows. The last claim holds because $k\sigma(k)\in K$.
\end{proof}

\begin{theorem}\label{thm:type6ad}
Let $G$ be of Type \Romannum{6} and let $d\mid 240|K|$ be written uniquely as $d=d_1d_2$ with $d_1\mid |K|$ and $d_2\mid 120$. Then
\[a_{d_1d_2}(G)=a_{d_1}(K)\,a_{d_2}\bigl(\operatorname{SL}_2(\mathbb{F}_5)\bigr)+b_{d_1}(\sigma)\,a_{d_2}\bigl(\operatorname{SL}_2(\mathbb{F}_5)\cdot S\bigr),\]
where $a_{d_1}(K)$ is given by Theorem~\ref{thm:adgfpf52}, $b_{d_1}(\sigma)$ by \eqref{equ:bdsigma}, and $a_{d_2}$ by Tables~\ref{tab:sl2f5_order} and~\ref{tab:coset_order}; for the remaining $d$ we have $a_d(G)=0.$
\end{theorem}
\begin{proof}
Since $G=G_1\sqcup G_1S$, we count the two parts separately. Let $k\in K$ and $x\in\operatorname{SL}_2(\mathbb{F}_5)$. As $|K|$ is coprime to $30$ and $\operatorname{ord}(x)\mid120$, we have $\operatorname{ord}(kx)=\operatorname{ord}(k)\operatorname{ord}(x)$, so the elements of $G_1$ of order $d$ are counted by $a_{d_1}(K)a_{d_2}(\operatorname{SL}_2(\mathbb{F}_5))$. For $h\in\operatorname{SL}_2(\mathbb{F}_5)\cdot S$, \eqref{lem:cosetorder} gives $\operatorname{ord}(kh)=\operatorname{ord}(k\sigma(k))\operatorname{ord}(h)$, and by Lemma~\ref{lem:coseto} we have $\operatorname{ord}(h)\in\{4,8,12\}$; hence the elements of $G_1S$ of order $d$ are counted by $b_{d_1}(\sigma)\,a_{d_2}(\operatorname{SL}_2(\mathbb{F}_5)\cdot S)$. Finally, $(|K|,240)=1$, so every divisor of $240|K|$ admits a unique decomposition $d=d_1d_2$ of the above form.
\end{proof}

\begin{example}\label{exa:type6}
Let $K=C_{29}\rtimes C_{49}=\langle A,B\mid A^{29}=B^{49}=1,\ BAB^{-1}=A^{16}\rangle$, so that $r=16$ has order $7$ in $(\mathbb{Z}/29\mathbb{Z})^{\times}$ and $(|K|,30)=1$. This is the smallest non-abelian $K$ occurring in Type \Romannum{6}: since $(|K|,30)=1$, every prime divisor $p$ of $d=\operatorname{ord}_m(r)>1$ satisfies $p^{2}\mid n$, so $p\ge7$, $n\geq49$, and $7\mid\varphi(m)$ forces $m\geq29$. Let $\sigma(A)=A^{-1}$ and $\sigma(B)=B$, that is, $s=28$ and $t=1$. Then $|K|=1421=7^2\cdot 29$, and \eqref{equ:bdsigma} gives
\[b_1(\sigma)=29,\qquad b_7(\sigma)=174,\qquad b_{49}(\sigma)=1218,\qquad b_{29}(\sigma)=b_{203}(\sigma)=b_{1421}(\sigma)=0,\]
while Theorem~\ref{thm:adgfpf52} gives $a_1(K)=1$, $a_7(K)=6$, $a_{29}(K)=28$, $a_{49}(K)=1218$ and $a_{203}(K)=168$. With Tables~\ref{tab:sl2f5_order} and~\ref{tab:coset_order} we obtain, for instance,
\[a_{4}(G)=1\cdot30+29\cdot20=610,\qquad a_{8}(G)=29\cdot60=1740,\]
\[a_{196}(G)=1218\cdot30+1218\cdot20=60900,\qquad a_{588}(G)=1218\cdot40=48720,\]
where $|G|=240\cdot1421=341040.$ Note that $\sum_{d\mid|K|}b_d(\sigma)=|K|$, which is consistent with $\sum_{d\mid 240|K|}a_d(G)=240|K|.$
\end{example}

Substituting $c_d(G)=a_d(G)/\varphi(d)$ into Theorems~\ref{thmnff35} and~\ref{thmpif36} and Corollary~\ref{cor:eulerg37}, we obtain closed-form expressions for the number $|F|$ of non-free points, the number $|\pi(F)|$ of singular points, and the orbifold Euler characteristic of $V_{\mathbb{C}}/(G\ltimes L)$ for every FPF group of Type \Romannum{6}. Together with Theorem~\ref{thm:type6ad}, this completes the explicit computation of the invariants of Section~\ref{sec:acts} for all six types.

\section*{Data availability}
No data were generated or analysed during this study.

\printbibliography

\end{document}